\documentclass[12pt]{article}
\usepackage{amsmath, amsthm, amssymb,amscd}

\title{Littlewood's Algorithm and Quaternion Matrices\footnotetext{This is the authors' version of a work that was published in Linear Algebra Appl. 298 (1999) 193--208.}}

\author{Dennis I. Merino\\
Department of Mathematics \\
Southeastern Louisiana University \\
Hammond, Louisiana 70402-0687 \\
dmerino@selu.edu \\
   \and
Vladimir V. Sergeichuk
\thanks{Partially supported by Grant No. UM1-314 of the
U.S. Civilian Research and Development Foundation for the
Independent States of the Former Soviet Union.}\\
Institute of Mathematics \\
Tereshchenkivska 3,
Kiev, Ukraine \\
sergeich@imath,kiev.ua}

\date{}

\begin{document}
\maketitle

\newtheorem{theorem}{Theorem}[section]
\newtheorem{lemma}{Lemma}[section]
\theoremstyle{remark}
\newtheorem{remark}{Remark}[section]

\def\Reals{\hbox{\rm I\kern-.18em R}}
\def\reals{\Reals}
\def\Complexes{\hbox{\rm C\kern-.43em
       \vrule depth 0ex height 1.4ex width .05em\kern.41em}}
\def\complexes{\Complexes}
\def\field{\hbox{\rm I\kern-.18em F}} 
\def\Naturals{\hbox{\rm I\kern-.17em N}}
\def\naturals{\Naturals}
\def\integers{\hbox{\rm Z\kern-.3em Z}}
\def\hh{\hrule height0.9pt width1.1em}
\def\vv{\vrule width0.8pt depth0.12em height0.92em}
\def\square{\vbox{\kern0.15em\hh\kern0.9em\hh\kern-1.05em
            \hbox{\vv\kern0.93em\vv}}}

\def\quat{\hbox{\rm I\kern-.17em H\kern-.17em \rm I}}
\def\poz{\hbox{\rm I\kern-.17em P}}

\def\newpic#1{%
   \def\emline##1##2##3##4##5##6{%
      \put(##1,##2){\special{em:point #1##3}}%
      \put(##4,##5){\special{em:point #1##6}}%
      \special{em:line #1##3,#1##6}}}
\newpic{}

\begin{abstract}
A strengthened form of Schur's triangularization theorem is given for
quaternion matrices with real spectrum (for complex matrices it was given
by Littlewood). It is used to classify projectors (${\cal A}^2={\cal A}$)
and self-annihilating operators (${\cal A}^2=0$) on a quaternion unitary
space and examples of unitarily wild systems
of operators on such a space are presented.
Littlewood's algorithm for reducing a complex matrix to a
canonical form under unitary similarity is extended to quaternion
matrices whose eigenvalues have geometric multiplicity 1.
\end{abstract}

\vspace{.125in}
\section{Introduction and Definitions}

We denote the set of $m$-by-$n$ matrices by $M_{m,n}(\field)$, where $\field = \complexes$
or $\field = \quat$, the skew field of real quaternions with involution
\[ \overline{a + bi + cj + dk} = a - bi - cj - dk, \mbox{  }
a, b, c, d \in \reals, \]
and write
$M_{n} \equiv M_{n,n}$;
$A^{\ast}$ denotes the conjugate transpose; the $n$-by-$n$ upper triangular Jordan block with eigenvalue $\lambda $ is denoted by
$J_{n}(\lambda)$.

A matrix $U\in M_{n}(\field )$ is called {\em unitary} if $U^{\ast}U = I$.
Two matrices $A$ and $B \in M_{n}(\field)$ are {\em unitarily similar} (over $\field $)
if there exists a unitary $U$ such that $A = U^{\ast}BU$; they are called {\em unitarily equivalent} if there exist unitary $U$ and $V$ such that $A = UBV$.

Let $A$ be a quaternion $n$-by-$n$ matrix;
$\lambda\in\quat $ is  a (right) eigenvalue of $A$ if there exists a nonzero
$v\in\quat^n$ such that $Av = v \lambda$. The eigenvalues are defined
only up to similarity: $Avh = vh \cdot h^{-1}\lambda h$ for each nonzero
$h\in\quat$, so $h^{-1}\lambda h$ is an eigenvalue of $A$ whenever
$\lambda$ is.
Every eigenvalue $\lambda=a+bi+cj+dk$ is similar to exactly one complex
number with nonnegative imaginary part, namely $a+\sqrt{b^2+c^2+d^2}i$
\cite[Lemma 2.1]{zhang}; this complex number is called a {\it standard eigenvalue} of $A$.
There exists a nonsingular $S \in M_{n}(\mbox{\quat})$ such that $ S^{-1}A S$ is a Jordan matrix
\begin{equation}
\label{eqtn1.2}
J = J_{n_{1}}(\lambda_{1}) \oplus \cdots \oplus J_{n_{k}}(\lambda_{k}), \quad
\lambda_j=a_j+b_ji\in\complexes,\ b_j\ge 0,
\end{equation}
with standard eigenvalues,
determined up to permutation of Jordan blocks \cite[Chapter 3]{jacobson}.
We will assume that
$\lambda_{1} \succeq \cdots \succeq \lambda_{k}$  with respect to the following ordering in $\complexes$:
\begin{equation}       \label{1.1}
a+bi\succeq c+di\ \ {\rm if\ either}\ \  a\ge c\ \ {\rm and}\ \ b=d,\ \ {\rm or}\ \ b>d.
\end{equation}
Performing the Gram-Schmidt orthogonalization on the columns of $S$ gives
a unitary matrix $U=ST$, where $T$ is an upper triangular matrix
with positive diagonal elements; this is
the $QR$ decomposition of $S$. Therefore, $A$ is unitarily similar to
an upper triangular matrix $U^{\ast}AU=T^{-1}JT$ having the same diagonal
as $J$ (Schur's theorem for quaternion matrices). For a survey of
quaternions and matrices of quaternions, see \cite{zhang}.
Canonical matrices of sesquilinear forms and pairs of
hermitian forms on a quaternion vector space,
and selfadjoint and isometric operators on a quaternion vector
space with indefinite scalar product, are given in \cite{sergeichuk2}.

This article is a result of attempts of the authors to extend to
quaternion matrices
Littlewood's algorithm \cite{littlewood} for reducing a complex matrix
to a canonical form under unitary similarity.
This algorithm was discussed in \cite{benedetti} and \cite{sergeichuk1};
see also \cite{sergeichuk} and the survey \cite{shapiro}.
Littlewood's algorithm is based on two statements:

(A) {\it Strengthened Schur Theorem.}
Each square complex matrix $A$ is unitarily
similar to an upper triangular matrix of the form
\begin{equation}                 \label{form2}
F=\left[\begin{array}{ccccc}
\lambda_{1}I_{n_{1}} & F_{12} & F_{13} & \cdots & F_{1s} \\
0 & \lambda_{2}I_{n_{2}} & F_{23} & \cdots & F_{2s} \\
0 & 0  & \lambda_{3}I_{n_{3}} & \ddots  & \vdots \\
\vdots & \vdots  & \vdots & \ddots & F_{s-1, s} \\
0 & 0  & 0 & \cdots & \lambda_{s}I_{n_{s}}
\end{array} \right],
\end{equation}
where $\lambda_1\succeq \cdots\succeq\lambda_s$ and
if $\lambda_i=\lambda_{i+1}$ then the columns of $F_{i,i+1}$  are
linearly independent; subject to the foregoing conditions,
the diagonal blocks $ \lambda_{i}I_{n_{i}}$ are uniquely determined by $A$.
If $F'$ is any other upper triangular matrix that is unitarily
similar to $A$ and satisfies the foregoing conditions, then
$F' = V^{\ast}FV$, where $V$ is complex unitary
and $V = V_{1} \oplus \cdots \oplus V_{s}$,
where each $V_{i} \in M_{n_{i}}(\complexes)$.
\footnote{
This formulation is not go with the algorithm: we reduce a matrix to the form (\ref{form2}), restrict the set of admissible transformations to those that preserve all diagonal and sub-diagonal blocks, then the preserving them matrices have the block-diagonal form. I propose the following version:
``where $\lambda_1\succeq \cdots\succeq\lambda_s$ and if $\lambda_i=\lambda_{i+1}$ then the columns of $F_{i,i+1}$  are linearly independent.
 The diagonal blocks $ \lambda_{i}I_{n_{i}}$ are uniquely determined by $A$. If $V^{-1}FV=F'$, where $V$ is complex unitary and $F'$ differs from $F$ only in over-diagonal blocks, then
$V = V_{1} \oplus \cdots \oplus V_{s}$,
where each $V_{i}$ is ${n_{i}}\times{n_{i}}$."
}

(B) {\it Singular value decomposition.}
Each nonzero complex matrix $A$ is unitarily equivalent
to a nonnegative diagonal matrix of the form
\begin{equation}                  \label{B}
D=a_1I_{n_{1}}\oplus\cdots\oplus a_{t-1}I_{n_{t-1}}\oplus 0,  \quad
a_i\in\reals, \ \ a_1>\cdots>a_{t-1}>0.
\end{equation}
If $U^{\ast}DV=D$, where $U$ and $V$ are complex unitary matrices,
then $U=U_1\oplus\cdots\oplus U_{t-1}\oplus U'$, $V=U_1\oplus\cdots
\oplus U_{t-1}\oplus V'$, and each $U_{i} \in M_{n_{i}}$.               \medskip

{\it Littlewood's algorithm.}
Let $A\in M_n(\complexes)$. Littlewood's algorithm has the following steps.
The first step is to reduce $A$ to the form (\ref{form2}); notice that
the diagonal blocks and all sub-diagonal blocks
of $F$ have been completely reduced.
We restrict the set
of unitary similarities to those that preserve the block structure of
(\ref{form2}),
\footnote{
 `Block structure' is not clear, the transformation must preserve all diagonal and sub-diagonal blocks.
}
 that is,
to the transformations
\begin{equation}       \label{alg}
F\mapsto V^{\ast}FV, \quad V=V_1\oplus\cdots\oplus V_s.
\end{equation}
The second step is to take the first nonzero superdiagonal block
$F_{ij}$ with respect to lexicographically
ordered indices, and reduce it to the form (\ref{B}) by unitary equivalence
$F_{ij}\mapsto V_i^{\ast}F_{ij}V_j$.  We make an additional partition of
$F$ into blocks conformal with the partition of the obtained block
$F'_{ij}=D$, and restrict the set of admissible transformations (\ref{alg})
to those that preserve $D$ (i.e., $V_i=U_1\oplus\cdots\oplus U_{t-1}\oplus U'$ and $V_j=U_1\oplus\cdots\oplus U_{t-1}\oplus V'$).
The $i$th step of the algorithm is to take the first block that changes
under admissible transformations and reduce it
by unitary similarity or equivalence to the form
(\ref{form2}) or (\ref{B}).  We restrict the set of admissible
transformations to those that preserve the reduced part and make
additional block partitions conformal to the block
that has just been
reduced. Since we have finitely many blocks, the process ends on a
certain matrix $A^{\infty}$ with the property that $A$ is unitarily
similar to $B$ if and only if $A^{\infty}=B^{\infty}$.
The matrix $A^{\infty}$ is called the {\it canonical form} of $A$ with respect to unitary similarity.

Statement (B) holds for all quaternion matrices, that is,
if $A$ is an $m$-by-$n$ quaternion matrix, then there exist
an $m$-by-$m$ unitary matrix $U$, an $n$-by-$n$ unitary matrix $V$, and
a diagonal matrix $\Sigma $ with nonnegative entries such that
$A = U \Sigma V $ \cite[Theorem 7.2]{zhang}.

In Section \ref{s2} we prove statement (A) for quaternion matrices with
real spectrum; it need not hold for quaternion matrices with nonreal
eigenvalues. The proof is based on modified Jordan matrices \cite{belitskii},
which we associate with the Weyr characteristic \cite[p. 73]{mcduffee} of a matrix.

In Section \ref{s3} we show that Littlewood's algorithm can be applied
to quaternion matrices with real spectrum that reduce
to the form (\ref{form2}) with $s=2$;
Littlewood's process then consists of at most two steps.
This two-step Littlewood's process can be used to obtain
the canonical forms of projectors (${\cal A}^2={\cal A}$) and
self-annihilating operators (${\cal A}^2=0$) on a quaternion unitary space.
A canonical form of a complex projector was given by
Dokovic  \cite{dokovic} and Ikramov  \cite{ikramov}; Ikramov's proof is based
on Littlewood's algorithm.

Unfortunately, Littlewood's algorithm cannot always be applied to
quaternion matrices with real spectrum that reduce to the
form (\ref{form2}) with $s \geq 3$.
The reason is that in the process of reduction,
one can meet a block with nonreal eigenvalues.
The problem of classifying such matrices has the same complexity as the problem of classifying all quaternion matrices up to unitary similarity since two quaternion matrices
$$
M_A=\left[\begin{array}{ccc}
3I_n & I_n & A \\
0 &2I_n & I_n  \\
0 & 0  & I_n
\end{array} \right]
\quad {\rm and} \quad
M_B=\left[\begin{array}{ccc}
3I_n & I_n & B \\
0 &2I_n & I_n  \\
0 & 0  & I_n
\end{array} \right]
$$
are unitarily similar if and only if $A$ and $B$ are unitarily similar.
That is, if $V$ is unitary and  $V^{\ast}M_AV\!$ $=M_B$, then
$V=V_1\oplus V_2\oplus V_3$
 by statement (A) for quaternion
matrices with real spectrum;
furthermore, $V_1=V_2=V_3$ and $V_1^{\ast}AV_1=B$.

Moreover, the problem of classifying quaternion matrices up to unitary similarity
(and even the problem of classifying, up to unitary
similarity, quaternion matrices with
Jordan Canonical Form $I\oplus 2I\oplus\cdots\oplus tI$)
has the same complexity as the problem of classifying an arbitrary system of linear
mappings on quaternion unitary spaces.
For example,  the problem of
classifying systems of four linear mappings
$$
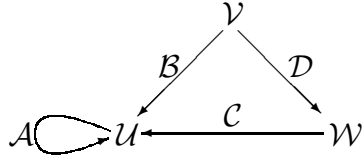
\unitlength 1.0mm
\linethickness{0.4pt}
\begin{picture}(46.34,20.33)
(2,17)
\put(17.34,20.33){\makebox(0,0)[cc]{$\cal U$}}
\put(46.34,20.33){\makebox(0,0)[cc]{$\cal W$}}
\put(31.34,36.33){\makebox(0,0)[cc]{$\cal V$}}
\put(19.34,20.33){\vector(-1,0){0.2}}
\put(43.34,20.33){\line(-1,0){24.00}}
\put(31.34,22.83){\makebox(0,0)[cc]{$\cal C$}}
\put(22.84,29.33){\makebox(0,0)[cc]{$\cal B$}}
\put(40.34,29.33){\makebox(0,0)[cc]{$\cal D$}}
\put(3.17,20.33){\makebox(0,0)[cc]{$\cal A$}}
\put(14.67,19.67){\vector(3,1){0.2}}
\bezier{64}(5.00,20.33)(5.00,15.33)(14.67,19.67)
\bezier{64}(5.00,20.00)(4.67,25.00)(15.00,20.67)
\put(30.00,34.00){\vector(-1,-1){11.33}}
\put(32.00,34.00){\vector(1,-1){11.00}}
\end{picture}
$$
($\cal{U, V, W}$ are arbitrary quaternion unitary spaces)
is the canonical form problem for quaternion matrices of the form
$$
M(A,B,C,D)=\left[\begin{array}{ccccc}
5I&I&A&C&B\\  0&4I&I&0&0\\   0&0&3I&0&0\\
0&0&0&2I&D \\ 0&0&0&0&I
\end{array}\right]
$$
under unitary similarity.  Indeed, by statement (A)
 for quaternion matrices
with real spectrum, if $V$ is unitary then $V^{\ast}M(A,B,C,D)V= M(A',B',C',D')$
implies $V=V_1\oplus\cdots\oplus V_5$.  It also follows that $V_1= V_{2} =V_3$, hence $(A,B,C,D)$ and $(A',B',C',D')$ are the matrices of the same system of linear mappings $(\cal{A,B,C,D})$ in different orthogonal bases of $\cal{U, V,W}$;
compare with \cite[Sect. 2.3]{sergeichuk}.

In particular, the problem of classifying quaternion matrices up to unitary similarity is equivalent to the problem of classifying $m$-tuples of quaternion matrices up to simultaneous unitary similarity
$$
(A_1,\dots,A_m)\mapsto (V^{-1}A_1V,\dots, V^{-1}A_mV).
$$
The case for $m$-tuples of complex matrices was proved in \cite{kruglyak}.
Other examples of classification problems that have the same complexity as
classifying arbitrary systems of linear operators on unitary spaces
are given in Section \ref{s3}.

In Section \ref{s4} we prove statement (A) for {\it nonderogatory}
quaternion matri\-ces{\em{---}}those matrices all of
whose eigenvalues have geometric multiplicity 1 \cite[Section 1.4.4]{hj1}.
We then extend Littlewood's algorithm to such matrices.
We also study the structure of their canonical matrices.

\section{A Strengthened Schur Theorem for Quaternion Matrices with Real Spectrum} \label{s2}

In this section we prove the following theorem.

\begin{theorem}
\label{theorem2.3}
Let $A$ be a given square quaternion matrix and suppose that $A$ has only
real eigenvalues.

(a) Then there exists a quaternion unitary matrix $U$
such that $F \equiv U^{\ast}AU$ has the form (\ref{form2}), where $\lambda_{1} \geq \lambda_{2} \geq \cdots \geq \lambda_{s} $ are
real numbers;
when $\lambda_{i} = \lambda_{i+1}$, then $n_{i} \geq n_{i+1}$, and
$F_{i, i+1}$ is an upper triangular matrix whose diagonal entries are positive real
numbers.

(b) The diagonal blocks $F_{ii} = \lambda_{i}I_{n_{i}}$ are uniquely
determined.  The off-diagonal blocks $F_{ij}$ are determined up to the
following equivalence.
If $V$ is a quaternion unitary matrix, then
 $F' \equiv V^{\ast}FV$ has
the form (\ref{form2}) with $F'_{ii} = F_{ii}$ (and without conditions on $F'_{i,i+1}$) if and only if
$V$ has the form
\[ V = V_{1} \oplus V_{2} \oplus \cdots \oplus V_{s}, \]
where each $V_{i}$ has size $n_{i}$-by-$n_{i}$.
\end{theorem}

The matrix (\ref{form2}) is a unitary variant of a modified Jordan matrix, which was proposed by Belitski\u{\i} \cite{belitskii} and is obtained from the Jordan matrix by a simultaneous permutation of rows and columns. We define it through the Weyr characteristic of a matrix.

A list of positive integers
$(m_{1}, m_{2}, \dots , m_{k})$ is said to  be {\em decreasingly ordered} if
$m_{1} \geq m_{2} \geq \cdots \geq m_{k}$.
Given a decreasing list $(m_{1}, m_{2}, \dots , m_{k})$, its
{\em conjugate} is the decreasingly ordered list $(r_{1}, r_{2}, \dots , r_{s})$
in which $s = m_{1}$ and $r_{i}$
is the number of $m_{j}$'s larger than or equal to $i$.

The Jordan Canonical Form $J_{m_{1}}(0) \oplus J_{m_{2}}(0) \oplus \cdots \oplus
J_{m_{k}}(0)$ of a nilpotent matrix $A$ can be arranged so that the sizes of its Jordan blocks
form a decreasingly ordered list
$m \equiv (m_{1}, m_{2}, \dots , m_{k})$, which is called the {\em Segre characteristic} of $A$;
its conjugate $r \equiv (r_{1}, r_{2}, \dots , r_{s})$ is called the {\em Weyr characteristic} of
$A$ \cite[p. 73]{mcduffee}.
Notice that $\mbox{rank}(A^{l}) = r_{l+1} + \cdots + r_{s}$ for $1 \leq l < s$.

\begin{lemma}
\label{lemma2.1}
Let $A \equiv J_{m_{1}}(0) \oplus J_{m_{2}}(0) \oplus \cdots \oplus
J_{m_{k}}(0)$ be given, and
suppose that $m_{1} \geq m_{2} \geq \cdots \geq m_{k}$.
Let $(r_{1}, r_{2}, \dots , r_{s})$ be the conjugate
of $(m_{1}, m_{2}, \dots , m_{k})$.  Then $A$ is similar to
\[ B \equiv \left[
\begin{array}{ccccc}
0_{r_{1}} & G_{12} & 0 & \cdots & 0 \\
0 & 0_{r_{2}} & G_{23} & \cdots & 0 \\
0 & 0  & 0_{r_{3}} & \ddots  & \vdots \\
\vdots & \vdots  & \vdots & \ddots & G_{s-1, s} \\
0 & 0  & 0 & \cdots & 0_{r_{s}}

\end{array} \right] \]
where $G_{i, i+1} \equiv \left[ \begin{array}{c} I_{r_{i+1}} \\ 0
\end{array} \right] $ is $r_{i}$--by--$r_{i+1}$.
\end{lemma}
\begin{proof}
Notice that
\[ \mbox{rank}(G_{i, i+1} G_{i+1, i+2} \cdots G_{i+t, i+t+1}) =
\mbox{rank}(G_{i+t, i+t+1}) = r_{i+t+1}. \]  One checks that
$\mbox{rank}(A^{l}) = \mbox{rank}(B^{l})$ for all $l$.  It follows that
$A$ is similar to $B$.
\end{proof}

\begin{remark} The two matrices $A$ and $B$ in Lemma {\ref{lemma2.1}} are permutation similar.
To get $B$ from $A$, permute the first columns of $J_{m_{1}}(0)$, $J_{m_{2}}(0)$,
\dots, and $J_{m_{k}}(0)$ into the first $k$ columns, then permute the
corresponding rows.  Next permute the second columns into the next columns and permute the
corresponding rows;
continue the process until $B$ is achieved.
\end{remark}

Let $A \in M_{n}(\mbox{\quat})$ be given, and let $J(A)$ be its Jordan Canonical Form (\ref{eqtn1.2}).  A repeated application of Lemma \ref{lemma2.1}
to the nilpotent part of $J(A) - \lambda_{j}I$ for each of the distinct eigenvalues $\lambda_{j}$ gives the following.

\begin{lemma}                  \label{theorem2.2}
Let $A \in M_{n}(\mbox{\quat})$ be given.
Then $A$ is similar to a unique matrix of the form
\begin{equation}
\label{form1}
B \equiv \left[
\begin{array}{ccccc}
\lambda_{1}I_{n_{1}} & G_{12} & 0 & \cdots & 0 \\
0 & \lambda_{2}I_{n_{2}} & G_{23} & \cdots & 0 \\
0 & 0  & \lambda_{3}I_{n_{3}} & \ddots  & \vdots \\
\vdots & \vdots  & \vdots & \ddots & G_{s-1, s} \\
0 & 0  & 0 & \cdots & \lambda_{s}I_{n_{s}}
\end{array} \right]
\end{equation}
with $\lambda_{1} \succeq \cdots \succeq \lambda_{s}$.
If $\lambda_{i} \not= \lambda_{i+1}$, then
$G_{i, i+1} = 0$;
otherwise, $n_{i} \geq n_{i+1}$ and
$G_{i, i+1} \equiv \left[ \begin{array}{c} I_{n_{i+1}} \\ 0
\end{array} \right] $ is $n_{i}$--by--$n_{i+1}$.
\end{lemma}

Belitski\u{\i} \cite{belitskii} called the matrix (\ref{form1}) a {\it modified Jordan matrix} and proved that all
matrices commuting with $B$ have an upper block-triangular form;
this fact plays a central role in his algorithm for reducing $m$-tuples
of complex matrices to a canonical form by
simultaneous similarity.
\medskip

{\it Proof of Theorem \ref{theorem2.3}.}
(a) Let $A \in M_{n}(\mbox{\quat})$ be given and suppose that $A$ has only
real eigenvalues, say $\lambda_{1} > \cdots > \lambda_{s}$.
Lemma \ref{theorem2.2} guarantees
that $S^{-1}AS =B$ for some nonsingular matrix $S$, and $B$ has the
form (\ref{form1}).  Perform
a Gram-Schmidt orthogonalization on the columns of $S$ so that
$U=ST$ is unitary and $T$ is an upper triangular matrix with positive diagonal elements.
$T^{-1}$ is necessarily upper triangular, and its diagonal elements
are also positive.

Write
\[
T^{-1} = \left[
\begin{array}{ccccc}
C_{1} & C_{12} & C_{13} & \cdots & C_{1s} \\
0 & C_{2} & C_{23} & \cdots & C_{2s} \\
0 & 0 & C_{3} & \ddots & \vdots \\
\vdots & \vdots  & \vdots &  \ddots & C_{s-1,s} \\
0 & 0 & 0  & \cdots & C_{s} \end{array} \right] \]
and
\[
T = \left[
\begin{array}{ccccc}
D_{1} & D_{12} & D_{13} &  \cdots & D_{1s} \\
0 & D_{2} & D_{23} & \cdots & D_{2s} \\
0 & 0 & D_{3} & \ddots & \vdots \\
\vdots & \vdots  & \vdots &  \ddots & D_{s-1,s} \\
0 & 0 & 0 & \cdots & D_{s} \end{array} \right]
\]
conformal to $B$.  Since each $\lambda_{i}$ is real,
direct computation of the product
$U^{\ast}AU = T^{-1}BT$
shows that

\[
U^{\ast}AU = T^{-1}BT = \left[
\begin{array}{ccccc}
\lambda_{1}I_{n_{1}} & F_{12} & F_{13} & \cdots & F_{1s} \\
0 & \lambda_{2}I_{n_{2}} & F_{23} & \cdots & F_{2s} \\
0 & 0 & \lambda_{3}I_{n_{3}} & \ddots & \vdots \\
\vdots & \vdots  & \vdots &  \ddots & F_{s-1,s} \\
0 & 0 & 0 & \cdots & \lambda_{s}I_{n_{s}} \end{array} \right] . \]
Since all the eigenvalues are real, the off-diagonal blocks
$F_{i, i+1}$ satisfy
\[ F_{i, i+1} = \lambda_{i}C_{i}D_{i, i+1} + C_{i}G_{i, i+1}D_{i+1}
+ \lambda_{i+1}C_{i, i+1}D_{i+1} . \]
Since $T^{-1}T = I$, we have
\[ C_{i} D_{i, i+1} + C_{i, i+1}D_{i+1} = 0. \]
Hence, when $\lambda_{i} = \lambda_{i+1}$,
\[ F_{i,i+1} = C_{i}G_{i, i+1}D_{i+1} . \]
If $\lambda_{i} = \lambda_{i+1}$, Lemma \ref{theorem2.2} guarantees that
$n_{i} \geq n_{i+1}$.  Moreover, the form of $G_{i, i+1}$
shows that $F_{i, i+1}$ is an upper triangular matrix whose
diagonal entries are positive real numbers.

(b) We now prove the uniqueness part.  That the eigenvalues of $F$ and their
multiplicity are determined is clear.  The sizes $n_{i}$ are also determined
by looking at powers of $F - \lambda_{i}I$.  We can also look at a decreasingly
ordered
list of the sizes of Jordan blocks corresponding to $\lambda_{i}$ and
notice that the conjugate of this list gives us the sizes needed.

Let $V$ be unitary.  Suppose that $F' \equiv V^{\ast}FV$ has the form
(\ref{form2}) and suppose further that $F'_{ii} = F_{ii}$.
We claim that $V$ is block diagonal conformal to $F$.

Write $V = \left[ V_{ij} \right]$ conformal to $F$ (and $F'$).  Form the
products $FV = VF'$.
Suppose that $\lambda_{s} \not= \lambda_{1}$.
The $(s, 1)$ block satisfies the equation
\[ \lambda_{s}V_{s1} = \lambda_{1} V_{s1} . \]
Hence, $V_{s1} = 0$.  If $\lambda_{s-1} \not= \lambda_{1}$, we look
at the $(s-1, 1)$ block and conclude that $V_{s-1, 1} = 0$.  We proceed
until $\lambda_{j} = \lambda_{1}$.

Now, we check if $\lambda_{s} \not= \lambda_{2}$.  If so, then we look
at the $(s, 2)$ block, and proceed as before.

We conclude that $V$ is block upper triangular, but since $V$ is unitary,
$V$ is block diagonal.  Hence, it suffices to prove the claim when all the
eigenvalues are the same, say $\lambda$.

As before, we write $V = \left[ V_{ij} \right]$ conformal to $F$, and look
at the equation $FV = VF'$.

The $(s, 1)$ block satisfies $\lambda V_{s1} = \lambda V_{s1}$.
However, the $(s-1, 1)$ block satisfies the equation
\[ \lambda V_{s-1, 1} + F_{s-1,s}V_{s1} = \lambda V_{s-1, 1} . \]
Hence, $F_{s-1,s}V_{s1} = 0$.  Since $F_{s-1,s}$ is upper triangular
with positive diagonal entries, $V_{s1} = 0$.

Next, we look at the $(s-2, 1)$ block to get
\[ \lambda V_{s-2, 1} + F_{s-2,s-1}V_{s-1, 1} = \lambda V_{s-2, 1} , \]
and similarly, we conclude that $V_{s-1, 1} = 0$.  Notice that the same
argument can be used to reach the conclusion that $V_{i,1} = 0$ for all
$i = 2, ..., s$.

We then look at the $(s, 2)$ block, $(s-1, 2)$ block, and so on.
The conclusion is that
$V$ is block upper triangular.  Since $V$ is also unitary, $V$ is
in fact block diagonal and the sizes of the blocks in $V$ match those
of $F$.

\section{Applications of the Strengthened Schur Theorem}                     \label{s3}

A square matrix $A$ is called a {\em projection} or
{\em idempotent} if $A^{2} = A$; it is called
{\em self-annihilating} if $A^{2} = 0$.  A canonical form of a complex
idempotent matrix under unitary similarity was obtained in \cite{dokovic, ikramov},
see also Section 2.3 of \cite{sergeichuk}.

\begin{theorem}
\label{theorem3.1}
(a) Let $A$ be a quaternion idempotent matrix ($A^2 = A$).  Then $A$
is unitarily similar to a direct sum that is uniquely
determined up to permutation of summands of matrices of the form
\[
\left[ \begin{array}{cc} 1 & b \\ 0 & 0 \end{array} \right]
\mbox{  (b positive), }
\left[ 1 \right], \left[ 0 \right]. \]

(b) Let $A$ be a self-annihilating quaternion matrix ($A^2 = 0$).  Then $A$ is unitarily
similar to a direct sum that is uniquely determined up to permutation
of summands of matrices of the form
\[
\left[ \begin{array}{cc} 0 & b \\ 0 & 0 \end{array} \right]
\mbox{ (b positive), } \left[ 0 \right]. \]

\end{theorem}

\begin{proof}
(a) Let $A$ be a quaternion idempotent matrix.  Theorem \ref{theorem2.3}(a)
ensures that $A$ is
similar to a matrix $F$ that has the form (\ref{form2}).  Since $A^{2} = A$,
we also have
$F^{2} = F$.  Hence, $F$ must have the form
\[
F = UAU^{\ast} =
\left[ \begin{array}{cc} I & F_{12} \\ 0 & 0 \end{array} \right]
. \]
By Theorem \ref{theorem2.3}(b), $F_{12}$ is determined up to unitary equivalence.
Now let $F_{12} =  V_{1}\Sigma V_{2}^{\ast}$ be the singular value decomposition
of $F_{12}$, where $V_{1}$ and $V_{2}$
are (quaternion) unitary matrices, $\Sigma = \mbox{diag}(b_{1}, \dots, b_{l})
\oplus 0$, and $b_{1} \geq \cdots \geq b_{l} > 0$.
Take $V \equiv V_{1} \oplus V_{2}$ and notice
that
\begin{equation}
\label{eqtn3.1}
V^{\ast}FV = \left[ \begin{array}{cc} I & \Sigma \\ 0 & 0 \end{array} \right].
\end{equation}
The conclusion follows by noting that the block matrix
(\ref{eqtn3.1}) is permutation similar to a matrix that is a sum of the desired matrices.

(b) The proof is similar to that of (a) except that $A^{2} = 0$ means that
\[
F = UAU^{\ast} =
\left[ \begin{array}{cc} 0 & F_{12} \\ 0 & 0 \end{array} \right]
, \] and $\Sigma = V_{1}^{\ast}F_{12}V_{2}$ has no zero columns.
\end{proof}

A self-annihilating quaternion matrix has the Jordan Canonical Form
$J_{2}(0) \oplus \cdots \oplus J_{2}(0) \oplus 0$ and a simple canonical
form under unitary similarity, as was shown in Theorem \ref{theorem3.1}(b).
What about quaternion matrices $A$ with the
Jordan Canonical Form $J_{2}(\lambda) \oplus \cdots \oplus
J_{2}(\lambda) \oplus \lambda I_{k}$?  If $\lambda $ is real, then
$(A - \lambda I)^{2} = 0$, and hence $A$ is unitarily similar to a
direct sum of matrices of the form
\[
\left[ \begin{array}{cc} \lambda & b \\ 0 & \lambda \end{array} \right]
\mbox{ ($b$ positive), } \mbox{ and } \left[ \lambda \right]. \]
However, when $\lambda \not\in \reals$, notice that
$(A - \lambda I)^{2}$ need not equal $0$.
Part (a) of the next theorem shows that the class of such matrices is
unitarily wild, that is, it contains the problem of classifying square
complex matrices up to (complex) unitary similarity and hence (see Section 1)
it has the same complexity as the problem of classifying arbitrary systems of
linear mappings on (complex) unitary spaces.
Parts (b)--(d) for complex matrices were given in
\cite{kruglyak1, sergeichuk}

\begin{theorem}
\label{theorem3.2}
The problem of classifying each of the following classes of matrices
and pairs of matrices under unitary similarity is
unitarily wild:

(a) square quaternion matrices whose Jordan Canonical Form consists only
of Jordan blocks $J_{2}(\lambda)$ and $ \left[ \lambda \right] $, where
$\lambda \not\in \reals$ is the same for all the matrices in the class;

(b) square quaternion matrices $A$ satisfying $A^{3} = 0$;

(c) pairs of quaternion idempotent matrices $(A, B)$, even if $A$ is self-adjoint,
that is $A^{2} = A^{\ast} = A$ and $B^{2} = B$;

(d) pairs of mutually- and self-annihilating quaternion matrices $(A, B)$,
that is $AB = BA = A^{2} = B^{2} = 0$.
\end{theorem}

\begin{proof}
(a) Let $\lambda \not\in \reals$ be a given eigenvalue, which
we may assume is standard, so
$\lambda = x + yi$ with $y > 0$.  To prove (a), we exhibit a mapping
$M \mapsto A_{M} \in M_{8n}(\quat)$ such that
$M, N \in M_{n}(\complexes)$ are (complex) unitarily similar if and only if
$A_{M}$ and $A_{N}$ are (quaternion) unitarily similar.

For such a given $\lambda$ and $M \in M_{n}(\complexes)$,
we define
\[ X_{M} \equiv \left[ \begin{array}{cccc}
4I_{n} & 0 & I_{n}j & Mj \\
0 & 3I_{n} & I_{n}j & I_{n}j \\
0 & 0 & 2I_{n} & 0 \\
0 & 0 & 0 & I_{n} \end{array} \right]   \in M_{4n}(\quat) \]
and
\[ A_{M} \equiv \left[ \begin{array}{cc}
\lambda I_{4n} & X_{M} \\
0 & \lambda I_{4n} \end{array} \right] \in M_{8n}(\quat) . \]
Notice that $A_M$ is similar to a direct sum of Jordan blocks $J_{2}(\lambda)$.

Now, suppose that $M$ is unitarily similar to $N$, say $U^{\ast}MU = N$ for
some unitary $U \in M_{n}(\complexes)$.  Let $V \equiv U \oplus U \oplus
\overline{U} \oplus \overline{U}$ and notice that
\[ (V \oplus V)^{\ast} A_{M} (V \oplus V) = A_{N} \]
since $j\overline{U} = Uj$.  Hence, $A_{M}$ is unitarily similar to $A_{N}$.

Conversely, suppose that $A_{M}$ is (quaternion) unitarily similar to $A_{N}$,
that is $V^{\ast}A_{M}V = A_{N}$ for some (quaternion) unitary matrix $V$.
We claim that $M$ and $N$ are (complex) unitarily similar.  Partition
the unitary matrix $V$ conformal to $A_{M}$, and rewrite the given
condition to get
\[ \left[ \begin{array}{cc}
\lambda I_{4n} & X_{M} \\
0 & \lambda I_{4n} \end{array} \right]
\left[ \begin{array}{cc}
V_{11} & V_{12} \\
V_{21} & V_{22} \end{array} \right]
=
\left[ \begin{array}{cc}
V_{11} & V_{12} \\
V_{21} & V_{22} \end{array} \right]
\left[ \begin{array}{cc}
\lambda I_{4n} & X_{N} \\
0 & \lambda I_{4n} \end{array} \right] , \]
which yields the following equalities:
\begin{itemize}
\item[(i)] $\lambda V_{11} + X_{M}V_{21} = V_{11} \lambda $,
\item[(ii)] $\lambda V_{12} + X_{M}V_{22} = V_{11}X_{N} + V_{12} \lambda $,
\item[(iii)] $\lambda V_{21} = V_{21} \lambda $,
\item[(iv)] $\lambda V_{22} = V_{21}X_{N} + V_{22} \lambda $.
\end{itemize}
Writing $\lambda = x + yi$, and using (iii) gives $i V_{21} = V_{21} i$.
It follows that $V_{21}$ has complex entries.

From (i),
we get $y(i V_{11} - V_{11} i) = - X_{M}V_{21}$.  Write
$V_{11} = P + Qj$, where $P$ and $Q$ have complex entries,
so that the equality becomes $2yQk = -X_{M}V_{21}$.
Write $Q = \left[ Q_{ij} \right] $ and $V_{21} = \left[ A_{ij} \right] $
conformal to $X_{M}$ to get
$$
2y \left[ \begin{array}{cccc}
Q_{11} & Q_{12} & Q_{13} & Q_{14} \\
Q_{21} & Q_{22} & Q_{23} & Q_{24} \\
Q_{31} & Q_{32} & Q_{33} & Q_{34} \\
Q_{41} & Q_{42} & Q_{43} & Q_{44} \end{array} \right] k =\qquad \qquad \qquad \qquad\qquad \qquad\qquad\qquad
$$
$$
\qquad\qquad\qquad
 -\left[ \begin{array}{cccc}
4I_{n} & 0 & I_{n}j & Mj \\
0 & 3I_{n} & I_{n}j & I_{n}j \\
0 & 0 & 2I_{n} & 0 \\
0 & 0 & 0 & I_{n} \end{array} \right]
\left[ \begin{array}{cccc}
A_{11} & A_{12} & A_{13} & A_{14} \\
A_{21} & A_{22} & A_{23} & A_{24} \\
A_{31} & A_{32} & A_{33} & A_{34} \\
A_{41} & A_{42} & A_{43} & A_{44} \end{array} \right] .
$$
Since $A_{ij}$ and $Q_{ij}$ have complex entries, we must have
$A_{4j} = 0 = A_{3j}$ for $j = 1, 2, 3, 4$.  Equating the
first two rows gives $V_{21} = 0$.  Moreover, $Q = 0$ as well,
which means that $V_{11}$ has only complex entries.

Since $V$ is unitary and since $V_{21} = 0$, we must also have $V_{12} = 0$.
Moreover, (iv) reduces to $\lambda V_{22} = V_{22} \lambda$, so
$V_{22}$ has only complex entries.

Now, (ii) reduces to $X_{M}V_{22} = V_{11}X_{N}$.  Write $X_{M} = P + Q_{M}j$
and $X_{N} = P + Q_{N}j$, where $P = \mbox{diag}(4I_{n}, 3I_{n}, 2I_{n}, I_{n})$,
and $Q_{M}$ and $Q_{N}$ have complex entries.  Since $V_{11}$ and $V_{22}$
have complex entries, we have $PV_{22} = V_{11}P$.
Multiplying this equality by $V_{22}^{\ast}P = PV_{11}^{\ast}$ gives
$P^{2}\equiv (PV_{22})(V_{22}^{\ast}P) = V_{11}P^{2}V_{11}^{\ast}$.  It follows
that $V_{11}$ is block diagonal; that is, it has the form
$V_{11} = C_{1} \oplus C_{2} \oplus C_{3} \oplus C_{4}$.  Similarly,
$V_{22}$ is block diagonal, and since $PV_{22} = V_{11}P$, we must
have $V_{22} = V_{11} = C_{1} \oplus C_{2} \oplus C_{3} \oplus C_{4}$.

Equating the noncomplex part of (ii) gives the equality
$Q_{M}j V_{22} = V_{11}Q_{N}j$.  Hence, we have the following
equalities:
\[ \overline{C}_{3} = C_{1}, \mbox{ } \overline{C}_{3} = C_{2},
\mbox{ } \overline{C}_{4} = C_{2}, \mbox{ and }
M\overline{C}_{4} = C_{1}N. \]
Therefore, $MC_{1} = C_{1}N$ and $M$ is unitarily similar to $N$.

 (b) Notice that using (a),
the problem of classifying square quaternion matrices up to
unitary similarity is unitarily wild.
Hence it suffices to prove that two $n\times n$ quaternion matrices
$M$ and $N$ are unitarily similar if and only if the two
$3n\times 3n$ matrices
$$
A_M=\left[\begin{array}{ccc}
  0 & I_n & M \\ 0 & 0 & I_n \\ 0&0&0
\end{array} \right]  \ \ {\rm and} \ \
A_N=\left[ \begin{array}{ccc}
0&I_n&N \\0 &0&I_n\\0&0&0
\end{array} \right]
$$
are unitarily similar.  One checks that
$A_M^3=A_N^3=0$.

Suppose $V^{\ast}A_MV=A_N$, where $V$ is unitary.
By Theorem \ref{theorem2.3}(b), $V$ has the form  $V_1\oplus V_2\oplus V_3$.
The equality $A_{M}V = VA_{N}$ now gives $V_1=V_2=V_3$, and
thus $V_1^{\ast}MV_1=N$.

(c) We look at the pairs of quaternion idempotent matrices
     $$ \left(\left[ \begin{array}{cc}
                       I_n &  0                   \\
                       0 &  0
              \end{array} \right],
              \left[ \begin{array}{cc}
                       M & I_n-M                  \\
                       M & I_n-M
              \end{array} \right]\right)
\ \  {\rm and} \ \
        \left(\left[ \begin{array}{cc}
                       I_n &  0                   \\
                       0 &  0
              \end{array} \right],
              \left[ \begin{array}{cc}
                       N & I_n-N                  \\
                       N & I_n-N
              \end{array} \right]\right),$$
which are unitarily similar if and only if $M$ and $N$ are unitarily similar.

(d) The pairs of quaternion matrices
$$
\left(\left[ \begin{array}{cc}
              0 &  I_n \\  0 &  0
              \end{array} \right],
   \left[ \begin{array}{cc}
           0 &  M \\  0 &  0
\end{array} \right]\right) \ \ {\rm and} \ \
        \left(\left[ \begin{array}{cc}
            0 &  I_n   \\   0 &  0
              \end{array} \right],
\left[ \begin{array}{cc}
      0 &  N  \\     0 &  0
     \end{array} \right]\right)
$$
are unitarily similar if and only if $M$ and $N$ are unitarily similar.
\end{proof}


\section{Littlewood's Algorithm for Nonderogatory Matrices}
\label{s4}

A square matrix is called {\it nonderogatory}
if every (standardized) eigenvalue has geometric multiplicity 1, that is, its Jordan Canonical Form does not contain two Jordan blocks
having the same standard eigenvalue \cite[Section 1.4.4]{hj1}.
In this section, we give an algorithm for reducing a
nonderogatory quaternion matrix $A$ by unitary similarity
to a certain matrix $A^{\infty}$, which has the property
that $A$ and $B$ are unitarily similar if and only if
$A^{\infty}=B^{\infty}$.
We call such a matrix $A^{\infty}$ the {\it canonical form of $A$
with respect to unitary similarity}.

We denote by
${\cal U}({\field})\equiv \{f\in {\field}\,|\,{\bar f}=f^{-1}\}$ the set of
unitary elements of ${\field}$, where $\field$ is  ${\quat}$,
$\complexes$, or $\reals$.

\begin{theorem}      \label{theorem4.1}
(a) Each nonderogatory quaternion matrix $M$ is unitarily similar to an upper triangular matrix of the form
\begin{equation}    \label{2.1}
A=\left[ \begin{array}{cccc}
\lambda_1& a_{12} & \cdots & a_{1n} \\
&\lambda_2 & \cdots & a_{2n}  \\
&& \ddots & \vdots  \\
0&&& \lambda_n
\end{array} \right],\ \ \
      \parbox{5cm}%
{$\lambda_l=x_l+y_li\in{\complexes},\ y_l\ge 0,\\
\lambda_1\succeq\cdots\succeq\lambda_n$ ,\\
$a_{l,l+1}\notin {\complexes}j$ if $\lambda_l= \lambda_{l+1}.$}
\end{equation}

     (b) The diagonal elements of $A$ are uniquely determined.
Moreover, for every quaternion unitary matrix $S$,
the matrices $A$ and $A'=S^{\ast}AS $ have the form (\ref{2.1})
if and only if $\lambda_{1}'= \lambda_{1},\dots, \lambda_{n}'= \lambda_{n}$ and
\begin{equation}    \label{2.1'}
S={\rm diag}(s_1, \dots, s_n),\quad
s_1, \dots, s_p\in {\cal U}(\complexes),\
s_{p+1}, \dots, s_n\in {\cal U}(\mbox{\quat}),
\end{equation}
where $p$ is such that  $\lambda_1, \dots, \lambda_p \not\in {\reals}$
and $\lambda_{p+1}, \dots, \lambda_n \in {\reals}$.
\end{theorem}

\begin{proof}
(a) The proof follows that of Theorem \ref{theorem2.3}(a),
that is, we write $S^{-1}MS = J$, where $J$ is the Jordan canonical form
of $M$; and all the eigenvalues lie in the upper half-plane.
We then apply the $QR$ factorization to $S$ to obtain
a unitary matrix $U = ST$ with $T$ an upper triangular matrix whose
diagonal elements are positive real numbers.
The matrix $A\equiv U^{\ast}MU=T^{-1}JT$ has the desired form (\ref{2.1}).  Notice
that when $\lambda_{l} = \lambda_{l+1}$, $a_{l, l+1}$ has the form
$-\lambda_{l}x + x\lambda_{l} + a$, with $a > 0$.  Now, if $\lambda \in
\complexes$ and $x \in \mbox{\quat}$, then $-\lambda x + x \lambda \in
\complexes j$.  Hence, when $\lambda_{l} = \lambda_{l+1}$, $a_{l, l+1}
\not\in \complexes j$.

(b) The proof also follows that of Theorem \ref{theorem2.3}(b),
and makes use of the techniques used in Theorem \ref{theorem3.2}(a).
First, notice that by the uniqueness of the Jordan form
and the fact that the eigenvalues are ordered we
must have $\lambda_{1}'= \lambda_{1},\dots, \lambda_{n}'= \lambda_{n}$.

Now, we show that $S$ is block diagonal.
If $\lambda_1 \not= \lambda_n$, then the $(n,1)$ entries of
$AS=SA'$ give $\lambda_ns_{n1}= s_{n1}\lambda_1$.  We express $s_{n1} = p + qj$,
with $p, q \in \complexes$ and conclude that
$s_{n1}=0$ since $\lambda_{1}$ and $\lambda_{n}$ are complex
numbers with nonnegative imaginary components.  Another
way  to look at it is that otherwise we would have
$s_{n1}^{-1}\lambda_n s_{n1}=\lambda_1$, contradicting
the fact that $\lambda_1 \not= \lambda_n$
and $\lambda_1$ and $ \lambda_n $ are complex numbers with
nonnegative imaginary components.

Next, we check if $\lambda_1 \not= \lambda_{n-1}$.
If so, then we look at the $(n-1, 1)$ entries of $AS=SA'$
to obtain $\lambda_{n-1}s_{n-1,1}= s_{n-1,1}\lambda_1$,
so that  $s_{n-1,1}=0$.  We proceed in this manner and conclude
that $s_{ij}=0$ whenever $i>j$ and $\lambda_i \not= \lambda_j$.

It follows that $S$ is upper block triangular, and since $S$ is unitary, it
must be block diagonal.  Hence, it suffices to prove that the claim holds
when all the eigenvalues coincide, that is,
$\lambda\equiv \lambda_1= \lambda_2= \cdots= \lambda_n$.

We consider two cases: $\lambda \in \reals$ and $\lambda \not\in \reals$.

Suppose $\lambda \in \reals$.
Notice that $\lambda s_{n1}= s_{n1}\lambda $ always holds.
Now, look at the $(n-1, 1)$ entries of $AS=SA'$ to obtain
$\lambda s_{n-1,1}+ a_{n-1,n}s_{n1}= s_{n-1,1}\lambda $.
Hence, $ a_{n-1,n}s_{n1}= 0$, and $ s_{n1}= 0$ since $a_{n-1,n}\ne 0$
by (\ref{2.1}). Now, look at the
$(n-2, 1)$ entries, then the $(n-3, 1)$ entries, and so on
and conclude that $s_{i1}=0$ whenever $i>1$.

Similarly, we look at the $(n,2)$ entries, the $(n-1, 2)$ entries, and so on
to conclude that in fact, $S$ is upper triangular.  Since $S$ is also
unitary, $S$ is also diagonal.

Now, suppose $\lambda \not\in \reals$. Then $\lambda= x+yi,\ y>0$;
the equality $\lambda s_{n1}= s_{n1}\lambda $ implies that
$is_{n1}= s_{n1}i$ and $s_{n1} \in \complexes$.
Furthermore, $\lambda s_{n-1,1}+ a_{n-1,n}s_{n1}= s_{n-1,1}\lambda $
implies $y(is_{n-1,1}- s_{n-1,1}i) + a_{n-1,n} s_{n1}=0$.
Write $s_{n-1,1}=p+qj$ and $a_{n-1,n}=u+vj$,
where $p, q, u,v \in \complexes$ to get
$2yqk+us_{n1}+v\bar{s}_{n1}j= us_{n1}+(v\bar{s}_{n1}+2yqi)j=0$.
Since $us_{n1}$ and $v\bar{s}_{n1}+2yqi $ are complex numbers,
and $u\not= 0$ (since $a_{n-1,n}\not\in {\complexes}j$ by (\ref{2.1})),
we must have $s_{n1}=0$ and $q=0$ (i.e., $s_{n-1,n}\in \complexes$).

Now, $\lambda s_{n-2,1}+ a_{n-2,n-1}s_{n-1,1}= s_{n-2,1} \lambda $
implies $s_{n-1,1}=0$ and $s_{n-2,1}\in \complexes$.
We repeat this process until we obtain $s_{ij}=0$ for all $i>j$
and $s_{ii}\in \complexes$. Since $S$ is unitary, $S$ is diagonal.  \end{proof}
\medskip

\noindent
\begin{center}
{\bf An algorithm for reducing a matrix $A$ of the form (\ref{2.1})
to canonical form with respect to unitary similarity}
\end{center}
\medskip

By Theorem \ref{theorem4.1}(b), the diagonal entries of $A$ are uniquely
determined.
Furthermore, all unitary similarity transformations that preserve the triangular form of $A$ and its diagonal entries have the form:
\begin{equation}       \label{2.2}
A \mapsto S^{\ast}AS,\
S\in {\cal G}_0\equiv  \underbrace{{\cal U}({\complexes})\times \cdots \times
{\cal U}({\complexes})}_{p}\times \underbrace{{\cal U}(\mbox{\quat})\times \cdots
\times {\cal U}(\mbox{\quat})}_{n-p}.
\end{equation}

We successively reduce the off-diagonal entries $a_{ij}\ (i<j)$
to a canonical form in the following order:
\begin{equation}       \label{2.3}
a_{12},\ a_{23},\dots,\ a_{n-1,n};\
a_{13},\ a_{24},\dots,\ a_{n-2,n};
\dots;\ a_{1n}.
\end{equation}
On each step, we use only those transformations (\ref{2.2})
that preserve the already reduced entries.

Suppose that all entries that precede $a_{lr}$ in the sequence (\ref{2.3})
have been reduced, and let all the transformations (\ref{2.2}) that
preserve the entries preceding $a_{lr}$ have the form
\begin{equation}       \label{2.4}
A \mapsto S^{\ast}AS,\quad
S\in {\cal G}\equiv
\{S={\rm diag}(s_1,\dots, s_n) \in{\cal G}_0\,|\, {\cal R}\},
\end{equation}
where $\cal R$ is a set of relations of the form
\begin{equation}       \label{2.5}
s_i\in{\complexes},\ s_i\in{\reals},\
s_i=s_j, \mbox{ or } s_i=s_j^{-1}\in{\complexes}.
\end{equation}
We reduce $a_{lr}$ to canonical form $a'_{lr}$
by transformations (\ref{2.4}) and show that all
transformations (\ref{2.4}) that preserve $a'_{lr}$ have the form
$A \mapsto S^{\ast}AS,\  S\in {\cal G}'=
\{S\in{\cal G}\,|\, \triangle {\cal R}\}$,
where $\triangle {\cal R}$ consists of relations of the
form (\ref{2.5}); this is required for the correctness of the induction step.

As follows from the form of relations (\ref{2.5}),
for every $i\in\{1,\dots,n\}$ there exists
${\field}_i\in \{\mbox{\quat},\, \complexes,\, \reals\}$ such that
$$
\{s_i\,|\,S\in{\cal G}\}={\cal U}({\field}_i).
$$
If $a_{lr}$ is not changed by transformations (\ref{2.4}),
we set $ a_{lr}'= a_{lr}$ and $\triangle {\cal R}=\emptyset $.

Denote by {\poz} the set of positive real numbers and
suppose that
\begin{equation}       \label{2.5'}
a_{lr}=z_1+z_2j,\quad z_1,z_2 \in \complexes,
\end{equation}
was changed by transformations (\ref{2.4}).
We have the following cases.

1) {\it ${\field}_l=\mbox{\quat}$ or ${\field}_r=\mbox{\quat}$. }
If $\cal R$ does not imply $s_l=s_r$, then we reduce $a_{lr}$
to the form $a'_{lr}=s^{-1}_la_{lr}s_r \in \mbox{\poz}$ and
obtain $\triangle {\cal R}=\{s_l=s_r\}.$
If $s_l=s_r$ follows from $\cal R$, then ${\field}_l={\field}_r=\mbox{\quat}$,
take $a_{lr}'=s^{-1}_la_{lr}s_l \in{\Complexes}$ with a nonnegative
imaginary component (note that $a_{lr}'\notin {\reals}$,
otherwise $a_{lr}$ is not changed by admissible transformations)
and obtain $\triangle {\cal R}= \{s_l\in\Complexes \}.$

2) {\it ${\field}_l={\field}_r={\Complexes}$ and
$\cal R$ does not imply $s_l=s_r$ or $s_l=s_r^{-1}$.}
Then by (\ref{2.5'})
$a_{lr}'=s_l^{-1}a_{lr}s_r= s_l^{-1}z_1s_r
+ (s_l^{-1}z_2{\bar s}_r)j = z_1s_l^{-1}s_r+ z_2s_l^{-1}s_r^{-1}j.$
If $z_1 z_{2} \not= 0$, we make $z_1' \in \mbox{\poz}$,
then $s_l=s_r$ (to preserve $z_1'$),
next make $z_2' \in \mbox{\poz}$, then  $s_l=s_r=\pm 1$;
we obtain $a_{lr}'=\mbox{\poz}1+ \mbox{\poz}j$ and $\triangle {\cal R}= \{s_l=s_r\in\Reals \}.$

If $z_1\ne 0= z_2$, we make $a_{lr}' \in \mbox{\poz}$ and obtain
$\triangle {\cal R}= \{s_l=s_r\}.$
If $z_1= 0$, then $z_2\not= 0$ (otherwise $a_{lr}$ is not
changed by admissible transformations) we make
$a_{lr}' \in \mbox{\poz}j$ and obtain $\triangle {\cal R}= \{s_l=s_r^{-1}\}.$

3) {\it ${\field}_l={\field}_r={\Complexes}$, $s_l=s_r$ or $s_l=s_r^{-1}$.}
If $s_l=s_r$, then $a_{lr}' = z_1+ z_2s_l^{-2}j,$
make $a_{lr}' \in \complexes + \mbox{\poz}j$ and obtain $\triangle {\cal R}= \{s_l\in\Reals \}.$

If $s_l=s_r^{-1}$, then $a_{lr}' = z_1 s_l^{-2}+ z_2j,$
make $a_{lr}' \in \mbox{\poz} + \complexes j$ and obtain $\triangle {\cal R}= \{s_l\in\Reals \}.$

4) {\it Either ${\field}_l={\Complexes}$ and  ${\field}_r={\Reals}$,
or ${\field}_l={\Reals}$ and  ${\field}_r={\Complexes}$.}
Make $a_{lr}'=1+z'_2j$ or $a_{lr}'=j$ and obtain
$\triangle {\cal R}= \{s_l=s_r\in \reals \}.$

5) {\it ${\field}_l={\field}_r={\Reals}$.}
Make $a_{lr}'=z_1'+z'_2j$ with
$z'_1\succ 0$ (see (\ref{1.1})), or
$a_{lr}'=z_2'j$ with $z'_2\succ 0$,
and obtain $\triangle {\cal R}= \{s_l=s_r \}.$
\medskip

The process ends with the reduction of $a_{1n}$.
We denote the matrix obtained by $A^{\infty}$;
it is the canonical form of $A$ with
respect to unitary similarity.
At each step we reduced an entry to a form that
is uniquely determined by the already reduced entries
and the class of (quaternion) unitarily similar to $A$, and so
we obtain the following theorem:

\begin{theorem}   \label{t2.1}
Two nonderogatory quaternion matrices $A$ and $B$ are
unitarily similar if and only if $A^{\infty}=B^{\infty}$.
\end{theorem}

For a canonical $n\times n$ matrix $A=A^{\infty}$,
its {\it graph} $\Gamma (A)$ has vertices $1,\dots, n$, and
$l$ and $r$ are jointed by an edge if and only if the relations
$s_l=s_r$ and $s_l=s_r^{-1}$ do not follow from the condition
of preserving the entries of $A$ that precede $a_{lr}$,
but one of them follows from the condition of preserving $a_{lr}$
(i.e., $s_l=s_r$ or $s_l=s_r^{-1}$ is contained in $\triangle {\cal R}$;
 see the cases 1--5).
Notice that there is an edge $(i,i+1)$ if $\lambda_i=\lambda_{i+1}$ in (\ref{2.1}) since then
$a_{i,i+1}\notin {\Complexes}j$.

A square matrix $A$ is called {\it unitarily indecomposable} if it is not
unitarily similar to a direct sum of square matrices.

\begin{theorem}    \label{t2.2}
(a) The graph of each canonical matrix is a union of trees.
Any union of trees with numbered vertices can be the graph of a canonical matrix.

(b) A canonical matrix is unitarily indecomposable if and only if its graph is a tree. Moreover, let the graph $\Gamma (A)$ of a canonical matrix $A$ be the union of $m$ trees $\Gamma_i\ (1\le i\le m)$ with the vertices $v_{i1}<v_{i2}<\cdots <v_{ir_i}$. Rearrange the columns of $A$ such that their old numbers form the sequence
$$
v_{11},\dots,v_{1r_1};\
v_{21},\dots,v_{2r_2};\,\dots\,;
v_{m1},\dots,v_{mr_m},
$$
then rearrange its rows in the same manner.
The matrix obtained has the form $A_1\oplus \cdots \oplus A_m$,
where each $A_i$ is a unitarily indecomposable canonical $r_i\times r_i$ matrix.
\end{theorem}

\begin{proof} (a) Let the graph $\Gamma (A)$ of a canonical matrix $A$
have a cycle $v_1$---$v_2$---$ \;\cdots\;$---$v_p$---$v_1$ ($p\ge 2$),
and let, say, $a_{\{v_1v_2\}},\dots, a_{\{v_{p-1}v_p\}}$
precede $a_{\{v_pv_1\}}$ in the sequence (\ref{2.3}), where $a_{\{ij\}}$
denotes $a_{ij}$ if $i<j$ and $a_{ji}$ if $i>j$.
Then the equality $s_{v_p}=s_{v_1}^{\pm 1}$ follows from the
condition of preserving the entries of $A$ that precede  $a_{\{v_pv_1\}}$,
a contradiction to the existence of the edge $v_p$---$v_1$.

Let a graph $\Gamma$ with vertices $1,\dots,n$ be a joint of trees. Take $A$ of the form (\ref{2.1}), in which $\lambda_1=ni,\ \lambda_2=(n-1)i, \dots, \lambda_n=i$, and, for every $l<r$, $a_{lr}=1$ if there is the edge $l$---$r$ and $a_{lr}=0$ otherwise. Clearly, $A$ is a canonical matrix and $\Gamma(A)= \Gamma$.

(b) Let $A$ be a canonical matrix. Since $a_{ij}=0$ whenever $i$ and $j$ are not connected in $\Gamma (A)$, the graph of a unitarily indecomposable canonical matrix is a tree.
It follows from the algorithm of reduction to canonical form that
if $\Gamma (A)$ is not a connected graph, then $A$ can be reduced
to a direct sum of unitarily indecomposable canonical matrices by
simultaneous permutation of its rows and columns.
\end{proof}

\end{document}